\documentclass[reqno,11pt]{amsart}
\usepackage[numbers]{natbib}
\usepackage{mathtools}
\usepackage[all]{xy}
\input epsf
\usepackage{lscape}

\calclayout

\numberwithin{figure}{section}

\theoremstyle{plain}
\newtheorem{thm}{Theorem}[section]
\newtheorem{lem}[thm]{Lemma}
\newtheorem{cor}[thm]{Corollary}
\newtheorem{prop}[thm]{Proposition}
\newtheorem{conj}[thm]{Conjecture}

\theoremstyle{definition}
\newtheorem{exl}[thm]{Example}
\newtheorem{defn}[thm]{Definition}

\DeclareMathOperator{\intt}{int}

\DeclareMathOperator{\eq}{eq}
\newcommand{\N}{\ensuremath{\mathbb{N}}}

\usepackage{setspace}

\title{The topological sliceness of 3-strand pretzel knots}
\author{Allison N.~Miller}
\address{Department of Mathematics, University of Texas, Austin, TX 78712, USA}

\begin{document}

\begin{abstract}
We give a complete characterization of the topological slice status of odd 3-strand pretzel knots, proving that an odd 3-strand pretzel knot is topologically slice if and only if either it is ribbon or has trivial Alexander polynomial. (By work of \cite{FS85}, a nontrivial odd 3-strand pretzel knot $K$ cannot both be ribbon and have $\Delta_K(t)=1$.) We also show that topologically slice even 3-strand pretzel knots (except perhaps for members of Lecuona's exceptional family of \cite{Lec13}) must be ribbon. 
These results follow from computations of the Casson-Gordon 3-manifold signature invariants associated to the double branched covers of these knots. 
\end{abstract}

\bibliographystyle{alpha}

\maketitle

\section{Introduction}

In the years since Fox first posed the Slice-Ribbon Conjecture (Problem 1.33 on Kirby's list  \cite{Kirby78}), its validity has been established for several families of knots. The usual strategy is to give an explicit list of ribbon knots in the family and then to provide an obstruction to the smooth sliceness of all others in the family. 
An early example of this is the following classification of the smoothly slice rational knots.
\begin{thm}[\cite{Lisca07}]
A rational knot is smoothly slice iff it is ribbon iff it is in $\mathcal{R}$. \end{thm}
Note that $\mathcal{R}$ is an explicit family of rational knots known to be ribbon at least since \cite{CG86}. Lisca argues that if $K$ is not in $\mathcal{R}$, then  Donaldson's diagonalization theorem obstructs $\Sigma_2(K)$ from  smoothly bounding a rational homology ball, and hence obstructs $K$ from being smoothly slice. 
Shortly thereafter, Greene and Jabuka used similar arguments along with additional obstructions coming from Heegaard Floer homology to completely classify the smoothly slice odd 3-strand pretzel knots.\footnote{Note that we call a pretzel knot $P(p_1, \dots, p_n)$  \emph{odd} if all of its parameters $p_i$ are odd and \emph{even} if (exactly) one parameter is even.}
\begin{thm}[\cite{GJ13}]\label{smoothpretzel}
Let $K$ be an odd 3-strand pretzel knot. Then $K$ is smoothly slice iff it is ribbon.
\end{thm}
The ribbon knots of the above theorem are the odd 3-strand pretzels that are, up to reflection, of the form $P(p,q,-q)$ or $P(1,q,-q-4)$ for some odd $p,q>0$. 
Note that both Lisca and Greene-Jabuka actually prove stronger results that completely characterize the order of rational knots and odd 3-strand pretzel knots in the smooth concordance group. Note that Theorem~\ref{smoothpretzel} recovers the following result of Fintushel and Stern.
\begin{thm}[\cite{FS85}]\label{toppretzel}
Let $K$ be a nontrivial odd 3-strand pretzel knot with $\Delta_K(t)=1$. Then $K$ is not smoothly slice. 
\end{thm}

Lecuona uses techniques analagous to those of Greene-Jabuka to describe the smooth sliceness of even 3-strand pretzel knots, except for an exceptional family $\{\pm P_a\}$. In fact, Lecuona's results are much broader, essentially characterizing the smooth sliceness up to mutation of all even pretzel knots not in this exceptional family.  
In particular, let $P_a= P(a,-a-2, - \frac{(a+1)^2}{2})$ for $a>0$ odd. It follows from work of Jabuka in \cite{Jab10} that the knots $\{ \pm P_a\}$ are exactly the even 3-strand pretzel knots with trivial rational Witt class and determinant one.
\begin{thm}[\cite{Lec13}] \label{lecuona}
Let $K$ be an even 3-strand pretzel knot
that is not of the form $\pm P_a$ for any $a \equiv 1, 11, 37, 47,49,59 \mod 60$. 
Then $K$ is smoothly slice iff it is ribbon. 
\end{thm}
The ribbon knots of the above theorem are the 3-strand pretzel knots that are of the form $P(p,q,-q)$ for some even $p$ and odd $q$. 

It is natural to ask to what extent these results, proved using the smooth machinery of Donaldson intersection form obstructions and Heegaard Floer homology, hold in the topological category. Theorem~\ref{toppretzel} implies that there are topologically but not smoothly slice 3-strand pretzel knots, but it is open whether smoothly slice equals topologically slice for rational knots. Note that recent work of Feller and McCoy shows that there are rational knots with distinct smooth and topological 4-genera \cite{FMc15}. 

We give an almost complete characterization of the topological sliceness of 3-strand pretzels via the computation of Casson-Gordon signatures corresponding to the double branched cover. In particular, we have the following two theorems. 

\begin{thm}\label{maintheorem}
Let $K$ be an odd 3-strand pretzel knot with nontrivial Alexander polynomial. Then $K$ is topologically slice iff $K$  is ribbon iff $K$ is of the form $\pm P(p,q,-q)$ or $\pm P(1,q, -q-4)$ for odd $p,q \in \mathbb{N}$. 
\end{thm}

\begin{thm}\label{evenpretzelthm}
Let $K$ be an even 3-strand pretzel knot that is not of the form $\pm P_a$ for $a \equiv 1, 11, 37, 47,59 \mod 60$.
Then $K$ is topologically slice iff $K$ is ribbon iff $K$ is of the form $P(p,q,-q)$ for some even $p$ and odd $q$. 
\end{thm}

Note that Theorems~\ref{maintheorem} and~\ref{evenpretzelthm}, when combined with Fintushel-Stern's result above (\cite{FS85}), recover Greene-Jabuka and Lecuona's results. We also have the following easy corollary. 
\begin{cor}\label{2bridge}
Let $K$ be a genus one alternating knot. 
Then $K$ is topologically slice iff $K$ is ribbon. 
\end{cor}
\begin{proof}
Let $K$ be a genus one alternating knot. Then by work of Stoimenow in \cite{Sto01}, $K$ is either an odd 3-strand pretzel knot with all parameters of the same sign (and hence has nonzero signature and is not even algebraically slice) or is rational. 
Therefore we may assume that $K$ is a genus one rational knot and hence (up to reflection) corresponds to the fraction $\frac{4ab+1}{2a}$ for some $a,b>0$. Note that $K$ has determinant $4ab+1>1$ and hence does not have trivial Alexander polynomial. Therefore, since such knots can also be described as the 3-strand pretzel knot $P(1,2a-1, -(2b+1))$, Theorem~\ref{maintheorem} implies that $K$ is ribbon.
\end{proof}

Lecuona conjectures that the (non)-existence of a Fox-Milnor factorization for the Alexander polynomial obstructs even the algebraic sliceness of the $\{\pm P_a\}$ family. When combined with Theorem~\ref{evenpretzelthm}, this would imply an affirmative answer to the following conjecture.

\begin{conj}
Let $K$ be an even 3-strand pretzel knot. Then $K$ is topologically slice iff $K$ is ribbon. 
\end{conj}

We can conveniently summarize Theorems~\ref{maintheorem} and~\ref{evenpretzelthm} in the following (slightly weaker) statement. 
\begin{thm}
Let $K$ be a 3-strand pretzel knot with nontrivial determinant. Then $K$ is topologically slice iff $K$ is ribbon. 
\end{thm}

A natural next question is the extent to which double branched cover Casson-Gordon signatures obstruct the topological sliceness of pretzel knots with more than three strands. 
However, several difficulties arise. First, pretzel knots with more than three strands have nontrivial mutations which often persist in concordance. (See \cite{HKL10} for examples.) However, even if we are willing to consider knots only up to mutation we cannot expect a complete answer from these techniques. In particular, there exist algebraically slice odd 5-strand pretzel knots with nontrivial Alexander polynomial but trivial determinant. (For example, consider $P(7,11,53,-5,-19)$.) There is no reason to believe that these knots are topologically slice, but there are also no double branched cover Casson-Gordon signatures to serve as sliceness obstructions.  


\section{Casson-Gordon signature invariants}

Casson and Gordon associate to a knot $K$ and a map $\chi:H_1(\Sigma_n(K))\to \mathbb{Z}_d$  the invariant $\tau(K,n, \chi) \in L_0(\mathbb{Q}(\omega)(t))\otimes \mathbb{Q}$. Note that $L_0(\mathbb{Q}(\omega)(t))$ is the Witt group  of non-singular Hermitian forms on finite-dimensional $\mathbb{Q}(\omega)(t)$-modules, where $\omega=e^\frac{2\pi i}{d}$. These invariants obstruct $K$'s topological sliceness as follows.

\begin{thm}[\cite{CG86}]\label{cgvanishingthm}
Let $K$ be a topologically slice knot and $n$ a prime power. 
Then there exists a square-root order subgroup $M \leq H_1(\Sigma_n(K))$, invariant under the action of the covering transformations, with the linking form of $\Sigma_n(K)$ vanishing on $M \times M$ (i.e. $M$ is a \emph{metabolizer} for the linking form) such that if $\chi$ is a prime-power order character  with $\chi|_{M}=0$, then $\tau(K, n, \chi)=0$.
\end{thm}

While this is a powerful sliceness obstruction, $\tau(K,n,\chi)$ cannot generally be directly computed. Instead, as originated in \cite{CG86}, one relates the signature\footnote{Note that $\bar{\sigma}_1:L_0(\mathbb{Q}(\omega)(t))\otimes \mathbb{Q} \to \mathbb{Q}$ is defined to be the average of the left and right limits of the signature of a representative form at $t=1$.} $\bar{\sigma}_1(\tau(K,n,\chi))$
 to a simpler signature associated to any three-manifold $Y$ and character from $H_1(Y)$ to a cyclic group. We now give the definition of this signature, following \cite{CG78}.

First, whenever  $X_\chi \to X$ is a cyclic $d$-fold cover, perhaps branched, we let $\omega= e^\frac{2\pi i}{d}$ and define the $\chi$-twisted homology of $X$ to be the $\mathbb{Q}(\omega)$ vector space
$ H_*^{\chi}(X):= H_*(C_*(X_\chi) \otimes_{\mathbb{Z}[\mathbb{Z}_d]} \mathbb{Q}(\omega))
\cong H_*(X_\chi) \otimes_{\mathbb{Z}[\mathbb{Z}_d]} \mathbb{Q}(\omega).
$

 We now let $Y$ be a closed 3-manifold and $\chi: H_1(Y) \to \mathbb{Z}_d$ an onto homomorphism. 
The map $\chi$ induces a $d$-fold cyclic cover $Y_\chi \to Y$ with a canonical generator $\tau$ for the group of covering transformations. Suppose that there is some $d$-fold branched cyclic cover of 4-manifolds $W_\chi \to W$ with branch set a closed surface $F \subset \intt(W)$ such that $\partial(W_\chi \to W)= r(Y_\chi \to Y)$ for some $r \in \N$.  Suppose also that the covering transformation $\tilde{\tau}$ of $W_\chi$ that induces rotation by $\frac{2\pi}{d}$ on the fibers of the normal bundle of the pre-image of $F$ in $W_\chi$ induces the canonical covering transformation $\tau$ on $Y_\chi$. We can always choose either $F=\emptyset$ or $r=1$ by bordism group considerations and an explicit description  in \cite{CG78}, respectively, and all of our work will be in one of these cases.  
The action of $\tilde{\tau}$ on $H:=H_2(W_\chi, \mathbb{C})$ allows us to decompose $H$  as the direct sum of eigenspaces $H_2^k(W_\chi)$ corresponding to eigenvalues $\omega^k$ for $k=0,\dots, d-1$. For $k>0$, define $\epsilon_k(W_\chi)$ to be the signature of the intersection form of $W_\chi$ when restricted to $H_2^k(W_\chi)$.\footnote{Note that $\epsilon_1(W_\chi)$ can be equivalently be defined as the signature of the twisted intersection form on $H_2^{\chi}(W)= H_2(W_\chi) \otimes_{\mathbb{Z}[\mathbb{Z}_d]} \mathbb{Q}(\omega)$.}

\begin{defn}
With the above set up, the $k^{th}$ Casson-Gordon signature of $(Y, \chi)$ is
\[
\sigma_k(Y,\chi)=\frac{1}{r}\left( \sigma(W)- \epsilon_k(W_\chi)- \frac{2k(d-k)}{d^2}([F]\cdot [F]) \right)
\]
\end{defn}
 Those familiar with the definition of $\tau(K,n,\chi)$ should note that we generally have $\sigma_1(S^3_n(K), \chi) \neq \bar{\sigma_1}(\tau(K, n, \chi))$. However, we can bound the difference between $\sigma_1(\Sigma_n(K),\chi)$ and $\bar{\sigma}_1(\tau(K,n,\chi))$, in a straightforward extension of Theorem 3 of \cite{CG86}. 

\begin{thm}[\cite{CG86}] \label{cgsigrelation}
Let $\chi:H_1(\Sigma_n(K)) \to \mathbb{Z}_d$ be an onto homomorphism. Then \[|\sigma_1(\Sigma_n(K), \chi)- \bar{\sigma}_1(\tau(K, n, \chi))| \leq \dim H_1^{\chi}(\Sigma_n(K))+1.\]
\end{thm}

Theorems~\ref{cgvanishingthm} and~\ref{cgsigrelation} immediately imply the following. 

\begin{cor}[\cite{CG86}]\label{mainobstructioncor}
Suppose that $K$ is a topologically slice knot and that $n=p^r$ is a prime power. 
Then there exists a metabolizer $M$ for the linking form on $H_1(\Sigma_n(K))$ such that if $\chi$ is any prime-power order character vanishing on $M$, then $|\sigma_1(\Sigma_n(K), \chi)| \leq \dim H_1^{\chi}(\Sigma_n(K))+1$. 
\end{cor}
Note that replacing $\chi$ with a nonzero multiple permutes the collection  $\{\sigma_k(Y,\chi)\}_{k=1}^{d-1}$, so the bound  of Corollary~\ref{mainobstructioncor} also holds for $\sigma_k(\Sigma_n(K),\chi)$.  If the obstruction of Corollary~\ref{mainobstructioncor} vanishes for characters from $H_1(\Sigma_2(K))$ to $\mathbb{Z}_d$, then we will refer to $K$ as \emph{CG-slice} at $d$.  The following proposition is often convenient in recognizing that $\Sigma_n(K)_\chi$ is a rational homology sphere, and hence that the bound of Corollary~\ref{mainobstructioncor} reduces to $|\sigma_1(\Sigma_n(K), \chi)| \leq 1$. 
 
\begin{prop}[\cite{CG78}] \label{qhs3recognition}
Suppose that $Y$ is a rational homology sphere with $H_1(Y, \mathbb{Z}_p)$ cyclic for some prime $p$. Then any cyclic $p^n$-fold cover of $Y$ is also a rational homology sphere. 
\end{prop}

In order to effectively apply this obstruction, we would like to  be able to compute $\sigma_k(Y, \chi)$ from an arbitrary integral surgery description of $Y$. 
 
\begin{defn}\label{twistcable}
Let $K$ be an oriented knot, and $A$ an embedded annulus such that $\partial A= K \sqcup -K'$ and $lk(K,K')=\lambda$. An \emph{$\lambda$-twisted $a$-cable} of $K$ is any oriented link $L$ obtained as the union of $n=n_+ + n_-$ parallel copies of $K$ in $A$ such that $n_+$ are oriented with $K$, $n_{-}$ opposite to $K$, and $n_{+}-n_{-}=a$. 
\end{defn}

Let $L= \bigcup_{i=1}^n L_i$ be an oriented link in $S^3$ such that surgery along $L$ with integer framings $\{\lambda_i\}_{i=1}^n$ gives $Y$. We refer to the meridian of component $L_i$ as $\mu_i$ and let $A= [a_{ij}]$ be the linking matrix of $L$. 
The following proposition is a generalization of Lemma 3.1 of \cite{CG78}. 
\begin{prop}[\cite{Gil81}] \label{computation1}
Let $Y$ be obtained by integer surgery on $L$ as above and $\chi: H_1(Y) \to \mathbb{Z}_d$ be an onto homomorphism.
Let $L_\chi$ be a satellite of $L$ obtained by replacing each $L_i$ by a non-empty $\lambda_i$-twisted $m_i$-cable of $L_i$, such that $\chi(\mu_i) \equiv m_i \mod d$.
Then for any $0<k<d$,
\[\sigma_k(Y,\chi)= \sigma(A)- \sigma_{L_\chi}(\omega^k) - \frac{2 k(d-k)}{d^2}\left(\sum_{i,j=1}^n m_i m_j a_{ij}  \right).\]
\end{prop}

In order to effectively apply Proposition~\ref{computation1} we will need to compute the Tristram-Levine signatures of cables of links. The techniques of colored signatures prove useful for this, as well as providing an independent means of computation for $\sigma_1(Y,\chi)$.

\section{Colored signatures of colored links}

A \emph{$n$-colored link} is an oriented link $L$ together with a surjective map assigning to each component of $L$ a color in $\{1, 2, \dots, n\}$. We let $L_i$ denote the sublink of $L$ consisting of $i$-colored components, and call each $L_i$ a \emph{colored component}. 
A \emph{C-complex} for a colored link $L$ consists of a union of Seifert surfaces for the colored components of $L$ which intersect only in a prescribed way (in `clasps'- see \cite{CF08} for the precise definition).

The \emph{colored signature} of $L$ is a map $\sigma_L:(S^1)^{n} \to \mathbb{Z}$ that is defined via the C-complex in a way exactly analagous to the definition of the Tristram-Levine signatures in terms of a Seifert surface for a link.  The colored signature shares many properties, including a 4-dimensional interpretation, with the ordinary signatures.  
We need the following results, due primarily to \cite{CF08}:
\smallskip

\noindent {\bf Recovery of Tristram-Levine signatures:}
Let $L$ be a $n$-component, $n$-colored link, and call the underlying ordinary link $L'$. 
 Then  for any $\omega~\in~S^1-~\{1\}$,
$\sigma_L(\omega, \dots, \omega)= \sigma_{L'}(\omega)+ \sum_{i<j} lk(L_i, L_j).$
\smallskip

\noindent {\bf Additivity:} Let $L'= L'_1 \cup \dots \cup L'_m$ and $L''=L''_{m+1} \cup \dots \cup L''_{m+n}$ be colored links and $L$ be the $(m+n-1)$-colored link obtained by connected summing any component of $L'_m$ with any component of $L''_{m+1}$. 
Then
$\sigma_L(\omega_1, \dots, \omega_{m}, \dots, \omega_{m+n-1})=
\sigma_{L}(\omega_1, \dots, \omega_{m})+ \sigma_{L''}(\omega_m, \dots, \omega_{m+n-1}).
$
\smallskip

\noindent{\bf Behavior under reversal and mirroring:}
The colored signature is invariant under global reversal of orientations. Also, letting $\bar{L}$ denote the mirror of $L$ we have 
$\sigma_{\bar{L}}(\omega_1, \dots \omega_n)= -\sigma_{L}(\omega_1, \dots, \omega_n).$
\smallskip

\noindent{\bf Behavior at 1}: (\cite{DFL14})
Let $L$ be an $n$-colored link and $L'$ be the $(n-1)$-colored link obtained by deleting the $n$th colored component of $L$. Then $\sigma_L(\omega_1, \dots, \omega_{n-1}, 1)= \sigma_{L'}(\omega_1, \dots, \omega_{n-1}).$
\smallskip

\noindent{\bf Hopf link computation:} Let $L$ be either Hopf link, considered as a 2-colored link. Then the colored signature function of $L$ is identically 0. 
\smallskip

We also need the following consequence of Degtyarev, Florens, and Lecuona's general description of the signature of a splice in \cite{DFL14}. 

\begin{exl}\label{coloredexample}
Let $L$ be the link shown in Figure~\ref{splicepicture}
\begin{figure}[h]
\includegraphics[height=3cm]{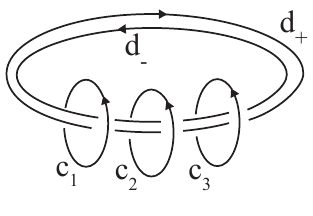}
\caption{A 5-colored link $L$}
 \label{splicepicture}
\end{figure}
and $\Phi(L)$ be the satellite of $L$ obtained by replacing each component $c_i$ with a coherently oriented torus link $T(a_i, p_i a_i)$ for $i=1,2,3$. 
Observe that as an ordinary oriented link, $L$ is isotopic to its mirror image in a way that swaps components $d_+$ and $d_-$ and preserves all other components. It follows that $\sigma_L(\omega_0, \omega_0, \vec{\omega})=0$ for all $\omega_0 \in S^1$ and $\vec{\omega}\in (S^1)^3$.
Let $\theta \in S^1$ be such that $\theta^{a_i} \neq 1$ for $i=1,2,3$. 
Then Theorem~2.2 of \cite{DFL14} and the above results imply that
 $\sigma_{\Phi(L)}(\theta)=\sum_{i=1}^3 \sigma_{T(a_i,p_i a_i)}(\theta)$ 
\end{exl}

Finally, in some cases colored signatures give us an alternate computational method for Casson-Gordon signatures.\footnote{Note that in the case that every meridian is sent to 1 and $k=1$, Theorems~\ref{computation1} and~\ref{coloredlinks} agree both with each other and with the original Lemma~3.1 of \cite{CG78}.
} 

\begin{thm}[\cite{CF08}]\label{coloredlinks}
Let $Y$ be a 3-manifold obtained by surgery on a framed $n$-component link $L$ with linking matrix $A=[a_{ij}]$.
Let $\chi: H_1(Y) \to \mathbb{Z}_d$ be a character of prime-power order that takes the meridian of each component of $L$ to a unit in $\mathbb{Z}_d$. Denote the lift of the image of the $i^{th}$ meridian of $L$ to $\{1, \dots, d-1\}$ by $m_i$. 
Consider $L$ as a $n$-colored link, and let  $\omega_{\chi}=(\omega^{m_1}, \dots, \omega^{m_n})$. 
Then 
\[ \sigma_1(Y, \chi)=
\sigma(A)
 -\left(\sigma_L(\omega_\chi)
-\sum_{i<j} a_{ij}\right) - 
\frac{2}{d^2}\left( \sum_{i,j}(d-m_i) m_j a_{ij}\right).\]
\end{thm} 

\section{Casson-Gordon signatures of 3-strand pretzels}\label{thework}

We now give the outline of the proof of Theorem~\ref{maintheorem}, deferring computations to later propositions. 

\begin{proof}[Proof of Theorem~\ref{maintheorem}]
Suppose that $K$ is an algebraically slice odd 3-strand pretzel knot with nontrivial Alexander polynomial. We will argue that either the Casson-Gordon signatures of $\Sigma_2(K)$ obstruct $K$'s topological sliceness or that $K$ is in fact ribbon. Since  $K$ is algebraically slice, the ordinary signature of $K$ vanishes, and so $pq+qr+pr<0$.
Also, $|H_1(\Sigma_2(K))|=-pq-qr-pr=D^2$ for some odd $D \in \mathbb{N}$. 
Note that since $K$ is a genus one algebraically slice knot with nontrivial Alexander polynomial, $D^2 \neq 1$ and hence $D$ has prime divisors.
Since $pq+pr+qr<0$, the parameters $p,q,$ and $r$ are not all of the same sign and so via reflection and the symmetries of 3-strand pretzel knots we can assume that $p,q>0$ and $r<0$.

In the following cases, the existence of a prime $d$ that satisfies the given conditions implies that the Casson-Gordon signatures of $\Sigma_2(K)$ corresponding to characters to $\mathbb{Z}_d$ obstruct $K$'s topological sliceness: 
\begin{enumerate}
\item $d$ divides $p$ and $q$ but not $r$: Proposition~\ref{dividespq}.
\item $d$ divides $r$ and exactly one of $p$ and $q$: Proposition~\ref{ddividespr}. 
\item $d$ divides all of $p$,$q$, and $r$: Proposition~\ref{dividespqr}
\item $d$ divides $D$ but none of $p$,$q$, and $r$; $p \not \equiv q \mod d$; and (assuming without loss of generality that $q>p$) $r \neq -(4p+q)$: Proposition~\ref{generaldneqdividepqr}. 
\item $d$ divides $D$ but none of $p$, $q$, and $r=-(4p+q)$: Proposition~\ref{exceptionalprop}. 
\item  $d$ divides $D$ but none of $p$,$q$, and $r$; $p\equiv q$ mod $d$; and $d \neq 3$: Proposition~\ref{chipischiqprop}.
\end{enumerate}

Now suppose that there is no prime satisfying any of the above. It follows that $p,q,$ and $r$ are relatively prime, $p \equiv q \mod 3$,  and $D$ is a power of three. We show that in this case the Casson-Gordon signatures corresponding to characters of order 3 and 9 obstruct topological sliceness in Proposition~\ref{powerof3}. 
\end{proof}

We now set up for our various computation. 
Since $K$ is  not ribbon, we have $r \neq -p,-q$.    We start with the surgery diagram for $\Sigma_2(K)$ given in Figure~\ref{surgerydiagram}, with linking matrix
$A= \left[ 
\begin{array}{cccc}
 0&1&1&1 \\
1&p&0&0 \\
1&0&q&0 \\
1&0&0&r
\end{array}
\right]$ and $\sigma(A)=0$. 
 We refer to the meridians of each component by $\mu_0, \mu_p, \mu_q,$ and $\mu_r$ according to their framings.
\begin{figure}[h]
\includegraphics[height=2cm]{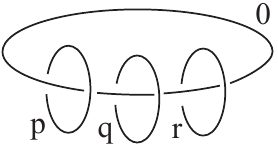}
\caption{A surgery diagram $L_0$ for $\Sigma_2(P(p,q,r))$.}
 \label{surgerydiagram}
\end{figure}
 
 Now let $d$ be any prime dividing $D$ and note that when $d$ does not divide all of $p,q,$ and $r$ we can easily check that $H_1(\Sigma_2(K), \mathbb{Z}_d)$ is cyclic and so every regular $d^n$-fold cyclic cover of $\Sigma_2(K)$ is a rational homology sphere (Proposition~\ref{qhs3recognition}). In addition, when $H_1(\Sigma_2(K),\mathbb{Z}_d)$ is cyclic any character $\chi: H_1(\Sigma_2(K)) \to \mathbb{Z}_d$ will vanish on any metabolizer for the linking form. (See Lemma 8.2 of \cite{HKL10}.) 
 So we have the following:
\newline
\noindent{\bf Useful Fact:} Suppose that $K=P(p,q,r)$ is topologically slice, $d$ is a prime dividing $pq+qr+pr$ that does not divide all of $p$, $q$, and $r$, and $\chi$ is any character $H_1(\Sigma_2(K)) \to \mathbb{Z}_d$. Then $|\sigma_1(\Sigma_2(K),\chi)|\leq 1$.

\subsection{Cases 1 and 2: $d$ divides some but not all of $p,q$ and $r$.}

\begin{prop}[Case 1]\label{dividespq}
Let $K=K(p,q,r)$ as above. Suppose that $d$ is a prime that divides $p$ and $q$ but not $r$. Then the Casson-Gordon signatures of $\Sigma_2(K)$ associated to characters to $\mathbb{Z}_d$ obstruct $K$'s topological sliceness. 
\end{prop}
\begin{proof}
We start by manipulating our surgery description for $\Sigma_2(K)$. Slide the curves with framing $p$ and $q$ over the curve with framing $r$. Then convert the $0$-framed 2-handle to a 1-handle, and cancel the 1-handle with the $r$-framed 2-handle.  We end with a new surgery description for $\Sigma_2(K)$ with underlying link $L= T(2,2r)$ and framings $p+r$ and $q+r$. The linking matrix of $L$ is $A= \left[ \begin{array}{cc} p+r & r \\ r & q+r\end{array} \right]$ and has $\sigma(A)=0$. Note that if we consider the entries of $A$ mod $d$ we get a presentation matrix for $H_1(\Sigma_2(K), \mathbb{Z}_d)$ with respect to basis $\{\mu_p, \mu_q\}$ which  immediately implies that $H_1(\Sigma_2(K), \mathbb{Z}_d) \cong \mathbb{Z}_d$, with generator $\mu_p=-\mu_q$.

By our useful fact, it suffices to show that for some $\chi: H_1(\Sigma_2(K)) \to \mathbb{Z}_d$ we have that $|\sigma_1(\Sigma_2(K),\chi)|>1$. Define $\chi$ on $H_1(\Sigma_2(K))$ by $\chi(\mu_p)=\chi(-\mu_q)=1$. So $L_\chi$ is the torus link $T(2, 2r)$ with strands oppositely oriented. Note that $\sigma_{\omega}(L_\chi)= -1$ and so we have by Proposition~\ref{computation1} that 
\[
\sigma_k(\Sigma_2(K), \chi)= 1-2((p+r)-2r+ (q+r))\frac{k(d-k)}{d^2}= 1-2\left(\frac{p+q}{d}\right)\left(\frac{k(d-k)}{d}\right)
\]
Note that $d$ divides $p$ and $q$, so $p+q\geq 2d$. Note that $k(d-k)\geq (d-1)$ for all choices of $k=1, \dots, d-1$. Since $d \geq 3$, we have
\[
| \sigma_k(\Sigma_2(K), \chi)| \geq 2 \cdot 2 \cdot \left(1-\frac{1}{3}\right)-1= \frac{8}{3}-1>1.\qedhere
\]
\end{proof}

Note that the above proof shows that $\sigma_k(\Sigma_2(K),\chi)<-1$ for all choices of $\chi:H_1(\Sigma_2(K))\to \mathbb{Z}_d$ and $k=1, \dots, d$, giving the following easy corollary.

\begin{cor}\label{subgroup}
For each odd prime $s$, let $K_s=P(p_s, q_s, r_s)$ be an odd 3-strand pretzel knot such that $p_s, q_s>0$ are divisible by $s$; $r_s<0$ is not divisible by $s$; and $p_sq_s+p_sr_s+q_sr_s=-s^2$. 
Then $\{K_s\}$ is a basis of algebraically slice knots for a $\mathbb{Z}^{\infty}$ subgroup of the topological concordance group. 
\end{cor} 
Note that such $K_s$ exist; for example, we can take $K_s=\left(s^2, s^2, -\frac{s^2+1}{2}\right)$. (Note that since $s$ is odd $s^2+1$ is equivalent to $2$ mod 4 and so this is an odd pretzel as desired.)

\begin{proof}
Suppose that $K= \sum_{i=1}^n a_i K_{s_i}$ is topologically slice, where each $a_i$ is nonzero. By reflecting $K$, we can assume without loss of generality that $a_1>0$. Since $K$ is topologically slice and $H_1(\Sigma_2(K), \mathbb{Z}_{s_i})$ is nonzero, it follows from Theorem~\ref{cgvanishingthm} that there is some nontrivial character $\chi: H_1(\Sigma_2(K)) \to \mathbb{Z}_{s_1}$ such that $\bar{\sigma}_1(\tau(K,2,\chi))=0$. 
Observe that 
\[H_1(\Sigma_2(K))= \bigoplus_{i=1}^n \left( 
H_1(\Sigma_2(K_{s_i}))^{\oplus |a_i|}
\right) = \bigoplus_{i=1}^n 
\left(\mathbb{Z}_{s_i}[t]/\langle t+1 \rangle\right)
^{\oplus |a_i|}. 
\]
Note that $\chi$ is nontrivial on each $H_1(\Sigma_2(K_{s_i}))$ factor for $i \neq 1$, and that $\chi$ can be decomposed as $\chi= \oplus_{j=1}^{|a_1|} \chi_j$, where each $\chi_j: H_1(\Sigma_2(K_{s_1})) \to \mathbb{Z}_{s_1}$ and at least one $\chi_j$ is nontrivial. By the additivity of Casson-Gordon signatures, $\bar{\sigma}_1(\tau(K,2,\chi))= \sum_{j=1}^{|a_1|} \bar{\sigma}_1(\tau(K_{s_1}, 2, \chi_j))$. However, the proof of Proposition~\ref{dividespq} shows that $\sigma_1(\Sigma_2(K_{s_1}), \chi_j)<-1$ whenever $\chi_j$ is nontrivial, and that
\[|\bar{\sigma}_1(\tau(K_{s_1},2,\chi_j)- \sigma_1(\Sigma_2(K_{s_1}),\chi_j)| \leq 1.\] 
It follows that $\bar{\sigma}_1(\tau(K,2,\chi_j))$ is strictly negative whenever $\chi_j$ is nontrivial (and zero when $\chi_j$ is trivial), and so that $\bar{\sigma}_1(\tau(K,2,\chi))<0$, which is our desired contradiction. 
\end{proof}

Now we continue to the next case. 
\begin{prop}[Case 2]\label{ddividespr}
Let $K=K(p,q,r)$. Suppose that there exists a prime $d$ that divides $r$ and exactly one of $p$ and $q$, but that $r \neq -p,-q$. Then the Casson-Gordon signatures of $\Sigma_2(K)$ associated to characters to $\mathbb{Z}_d$ obstruct $K$'s topological sliceness. 
\end{prop}

\begin{proof}
The argument is exactly analagous to that of the proof of Proposition~\ref{dividespq}, except that we choose $k$ to be $\frac{d-1}{2}$; the details are left to the reader. 
%
 \end{proof}
 
 \subsection{Case 3: $d$ divides all of $p$, $q$, and $r$}

In this case, we have that $H_1(\Sigma_2(K), \mathbb{Z}_d) \cong \mathbb{Z}_d \, \oplus \, \mathbb{Z}_d$, and so there may be metabolizers $M \leq H_1(\Sigma_2(K))$ with nontrivial image in $H_1(\Sigma_2(K), \mathbb{Z}_d)$. 
For each such metabolizer we provide a character $\chi$ to $\mathbb{Z}_d$ vanishing on $M$ such that the corresponding Casson-Gordon signature has sufficiently large absolute value. 
We first determine what ``sufficiently large'' is in the context of Corollary~\ref{mainobstructioncor}. 
\begin{lem}\label{h1boundlemma}
Let $\chi:H_1(\Sigma_2(K))\to \mathbb{Z}_d$. Then $\dim H_1^{\chi}(\Sigma_2(K))$ is 1 if $\chi(\mu_p), \chi(\mu_q),$ and $\chi(\mu_r)$ are all nonzero and 0 otherwise.
\end{lem}
\begin{proof} 
By slight simplifications of the Wirtinger presentation, we obtain $ \pi_1(S^3-L_0)= \langle \mu_0, \mu_p, \mu_q, \mu_r : \mu_0 \mu_p=\mu_p \mu_0, \mu_0 \mu_q=\mu_q \mu_0, \mu_0 \mu_r=\mu_r \mu_0  \rangle$, where $\mu_*$ is any meridian of the $*$-framed curve, for  $*=0,p,q,r$. Note that the 0-framed longitudes of the surgery curves are given with respect to this generating set by $\lambda_0= \mu_r\mu_q\mu_p$ and $\lambda_p=\lambda_q=\lambda_r=\mu_0$. 
 Gluing in solid tori according to the surgery framings gives new relations $\lambda_0= \mu_r \mu_q \mu_p= 1$, $\mu_p^p \lambda_p= \mu_p^p \mu_0=1$, $\mu_q^q \lambda_q= \mu_q^q \mu_0=1$, and $\mu_r^r \lambda_r= \mu_r^r \mu_0=1$, and hence we have the following presentation for $\pi_1(\Sigma_2(K))$:
 \begin{align*}
 \pi_1(\Sigma_2(K))&= 
 \left\langle
 \mu_0, \mu_p, \mu_q, \mu_r : 
\begin{array}{ll} 
 [\mu_0, \mu_p]=[\mu_0, \mu_q]=[\mu_0, \mu_r]=1 \\
 \mu_r\mu_q\mu_p=\mu_p^p \mu_0=\mu_q^q \mu_0=\mu_r^r \mu_0=1
 \end{array}
 \right\rangle \\
 &= \left\langle
 \mu_p, \mu_q, \mu_r : 
 \mu_r\mu_q\mu_p=1, \mu_p^p=\mu_q^q=\mu_r^r
 \right\rangle
 \end{align*}
 
Any choice of $x,y,z \in \mathbb{Z}_d$ such that $x+y+z \equiv 0 \mod d$ will define a character $\chi$ via $\mu_p \mapsto x$, $\mu_q \mapsto y$, and $\mu_r \mapsto z$. First suppose that none of $x,y,$ and $z$ are equivalent to $0$. Then by replacing $\chi$ with a nonzero multiple, which does not change the underlying cover, we may assume that $x=1$.

We apply the Reidemeister-Schreier algorithm\footnote{We lift 
loop $\mu_a$ to arcs $a_1, \dots, a_d$, 
loop $\mu_b$ to arcs $b_1, \dots, b_d$, and 
loop $\mu_c$ to arcs $c_1, \dots, c_d$ (all arcs are labeled by their starting point). We then contract $a_2, \dots, a_d$ and let $a=a_1$.} to our above presentation for $\pi_1(\Sigma_2(K))$ to obtain a presentation for $\pi_1(\Sigma_2(K)_\chi)$. Abelianizing, we obtain a presentation for $H_1(\Sigma_2(K)_\chi)$
with generators $a, b_1, \dots, b_d, c_1, \dots, c_d$ and relations
$a+ b_1 + c_x=0,\, b_k + c_{x+k-1}=0 \text{ for }k=2, \dots, d,$ and $
\frac{p}{d} a=\frac{q}{d}\left( b_1+ \dots+ b_d\right)= \frac{r}{d}(c_1 + \dots+ c_d)
$, where all subscripts are taken mod $d$.
This simplifies to 
\[
H_1(\Sigma_2(K)_\chi)= 
\langle
a, b_1, \dots, b_d 
:
\frac{p}{d}a=\frac{q}{d}(b_1+ \dots+ b_d)
=-\frac{r}{d}(b_1+ \dots + b_d+a)
\rangle
\]
So $ H_1(\Sigma_2(K)_\chi, \mathbb{Q})= 
\langle
b_1, \dots, b_d 
:
\left( pq+pr+qr \right)(b_1+ \dots b_d)=0
\rangle.$
Note that the covering transformation sends $b_i$ onto $b_{i+1}$ for $i=1, \dots, d-1$, and so $\dim H_1^{\chi}(\Sigma_2(K))=1$. 

When one of $x,y,$ and $z$ is 0, an extremely similar argument shows that $\Sigma_2(K)_\chi$ is a rational homology sphere and so $\dim H_1^{\chi}(\Sigma_2(K))=0$. 
\end{proof}

By considering the linking matrix $A$ for $L_0$ with its entries taken mod $d$, we see that $H_1(\Sigma_2(K), \mathbb{Z}_d)$ is generated as a $\mathbb{Z}_d$-module by the images of $\mu_p, \mu_q$ and $\mu_r$ (which we continue to refer to as $\mu_p, \mu_q, \mu_r$ by a mild abuse of notation) and has single relation $\mu_p+\mu_q+\mu_r=0$.  
Suppose that $\chi: H_1(\Sigma_2(K)) \to \mathbb{Z}_d$ sends $\mu_p$ to $a$, $\mu_q$ to $b$, and $\mu_r$ to $c$, where $0<a,b,c<d$. We must have $\chi(\mu_0)\equiv 0$ and $a+b+c\equiv 0$ mod $d$. We will use Proposition~\ref{computation1} to compute $\sigma_1(\Sigma_2(K),\chi)$,  
letting $L_\chi$ be the distant union of $T(a, pa)$, $T(b,qb)$, and $T(c, rc)$, each with all strands coherently oriented, along with two incoherently oriented linking 0 strands parallel to $\lambda_0$, as in Figure \ref{linkchi}. 
\begin{figure}[h]\label{linkchi}
\includegraphics[height=4.5cm]{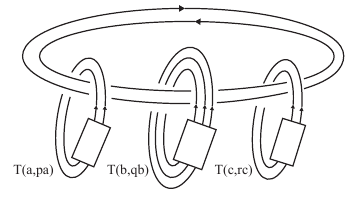}
\caption{The link $L_\chi$, pictured with $a=2, b=3, c=2$.}
\end{figure}

Note that as computed in Example~\ref{coloredexample}, $\sigma_{L_{\chi}}(\omega)= \sigma_{T(a, pa)}(\omega)+\sigma_{T(b,qb)}(\omega)+\sigma_{T(c,rc)}(\omega)
$. Also, Litherland's formula of \cite{Lith79} for the Tristram-Levine signature of a torus link\footnote{Note that while Litherland's result is stated only for torus knots, it holds for torus links as well. In particular, the  underlying computation in \cite{Brie66} of the signature of the Brieskorn manifold $V(p,q,r)_{\delta}= \{(z_1, z_2, z_3) \in \mathbb{C}^3 : z_1^p+ z_2^q+z_3^r= \delta\} \cap \mathbb{D}^6$ does not depend on any relative primeness of the parameters $p,q,$ and $r$.}
 implies that $\sigma_{T(j,jkn)}(e^{2\pi i/n})= -2j(j-1)k$ for $0<j<n$.  Therefore, we have that
\begin{align*}
\sigma_1(\Sigma_2(K), \chi)&= 0 - \sigma_{L_{\chi}}(\omega) - 2(a^2p+b^2q+c^2 r)\left(\frac{d-1}{d^2}\right)\\
&=-\sigma_{T(a, pa)}(\omega)-\sigma_{T(b,qb)}(\omega)-\sigma_{T(c,rc)}(\omega)- 2(a^2p+b^2q+c^2r)\left(\frac{d-1}{d^2}\right)  \\
&= 2a(a-1)\frac{p}{d}+ 2b(b-1)\frac{q}{d}+ 2c(c-1)\frac{r}{d} -2(a^2p+b^2q+c^2r)\left(\frac{d-1}{d^2}\right)\\
&=\frac{2}{d^2} \left( a(d-a)p +b(d-b)q+c(d-c)r \right)
\end{align*}

Unfortunately, we cannot conclude that $|\sigma_1(\Sigma_2(K), \chi)|> 1$ for all such choices of $\chi$. For example, when $K=P(3\cdot 7, 5\cdot 7, -17\cdot 7)$, $d=7$, and $\chi$ sends  $\mu_p$ to $2$, $\mu_q$ to $4$, and $\mu_r$ to $1$ we have $|\sigma_1(\Sigma_2(K), \chi)|= 8/11$. However, this choice of $\chi$ does not vanish on any metabolizer for the linking form $lk: H_1(\Sigma_2(K)) \times H_1(\Sigma_2(K)) \to \mathbb{Q}/\mathbb{Z}$, and so there is still some hope to obstruct $K$'s sliceness via double branched cover Casson-Gordon signatures.

\begin{lem}\label{evenintegerlemma}
Suppose $M$ is a metabolizer for the linking form on $H_1(\Sigma_2(K))$ with nonzero image in $H_1(\Sigma_2(K), \mathbb{Z}_d)$. 
If $\chi: H_1(\Sigma_2(K)) \to \mathbb{Z}_d$ vanishes on $M$ and takes $\mu_p, \mu_q, \mu_r$ to nonzero elements of $\mathbb{Z}_d$, then  $\sigma_1(\Sigma_2(K), \chi)$ is an integer that is divisible by 4. 
\end{lem}

\begin{proof}
For convenience, we write $p=dp', q=dq', r=dr'$. Note that we have assumed that $M$ has nontrivial image in $H_1(\Sigma_2(K), \mathbb{Z}_d)$, and hence we can assume that there is $\alpha=x\mu_p+ y \mu_q \in M$ such that not both of $x$ and $y$ are equivalent to 0 mod $d$. 

The linking form is given with respect to our $\mu_0, \mu_p, \mu_q, \mu_r$ generating set for $H_1(\Sigma_2(K))$ by $-A^{-1}$ (Gordon-Litherland). Direct computation shows that 
$lk(x \mu_p+ y \mu_q, x\mu_p+y\mu_q)= \frac{1}{D^2}((q+r)x^2-2rxy + (p+r)y^2)$.
Since $\alpha \in M$, we know that $D^2$ and hence $d^2$ divides $(q+r)x^2-2rxy+(p+r)y^2$, and so we have Equation $(*)$: $
(q'+r')x^2-2r'xy+(p'+r')y^2 \equiv 0 \mod d.
$

Now, let $\chi: H_1(\Sigma_2(K)) \to \mathbb{Z}_d$ be a character vanishing on $M$. As usual, we write $a= \chi(\mu_p), b= \chi(\mu_q), c= \chi(\mu_r)$, with $a+b+c \equiv 0 \mod d$. Since $\chi(\alpha)=ax+by\equiv 0 \mod d$, we can write $y=-a\bar{b}x$, and so neither $x$ nor $y$ is equivalent to $0$ mod $d$. Substituting into $(*)$, we obtain
\begin{align*}
0 &\equiv (q'+r')x^2-2r'xy+(p'+r')y^2 \\
&\equiv (q'+r')x^2+2r'a\bar{b}x^2+(p'+r')a^2\bar{b}^2 x^2 \\
& \equiv \left[a^2 \bar{b}^2 p' + q' + (a \bar{b}+1)^2 r'\right]x^2 \mod d. 
\end{align*}
Multiplying through by $(b^2/x^2)$ and recalling that $c^2 \equiv (a+b)^2 \mod d$ gives us 
that $a^2p' + b^2 q'+ c^2 r' \equiv 0  \mod d.$
Finally, we can write 
\begin{align*}
\frac{d^2}{2}\sigma_1(\Sigma_2(K), \chi) &= a(d-a)p +b(d-b)q+c(d-c)r\\
&= d( a(d-a)p' +b(d-b)q'+c(d-c)r') \\
&= d^2(p'+q'+r')- d(a^2p'+ b^2q'+c^2r'). 
\end{align*}
Observe that the right side is divisible by $d^2$, and hence $\sigma_1(\Sigma_2(K))$ is an integer. Also, since $d$ is odd, $a(d-a)p+b(d-b)q+c(d-c)r$ is even for any choice of $a,b,$ and $c$ and $\sigma_1(\Sigma_2(K),\chi)$ is divisible by 4. 
\end{proof}

\begin{prop}[Case 3]\label{dividespqr}
Let $K=P(p,q,r)$, with $p,q\neq-r$ and suppose that $d$ divides all of $p,q,r$. Then the Casson-Gordon signatures of $\Sigma_2(K)$ associated to characters to $\mathbb{Z}_d$ obstruct $K$'s topological sliceness. 
\end{prop}

\begin{proof}
Suppose  that $K$ is CG-slice at $d$, for a contradiction.
So there exists a metabolizer $M \leq H_1(\Sigma_2(K))$ such that any character $\chi$ of prime power order that vanishes on $M$ has $|\sigma_1(\Sigma_2(K), k\chi_0)| \leq \dim H_1^{\chi}(\Sigma_2(K))+1$ for all $0<k<d$. 
If there exists $\chi$ to $\mathbb{Z}_d$ vanishing on $M$ that takes 
any of $\mu_p,\mu_q$, and $\mu_r$ to $0$, then $\Sigma_2(K)_\chi$ is a rational homology sphere and arguments as in Cases 1 and 2 show that there is some $k$ such that $|\sigma_1(\Sigma_2(K), k\chi_0)|>1$.
 
So we can now assume that no such $\chi$ exists. In particular, this implies that the image of $M$ in $H_1(\Sigma_2(K), \mathbb{Z}_d)$ is nontrivial. So let $\chi_0: H_1(\Sigma_2(K)) \to \mathbb{Z}_d$ be a nontrivial character vanishing on $M$ and taking none of $\mu_p,\mu_q$, and $\mu_r$ to 0. Since $K$ is CG-slice, $|\sigma_1(\Sigma_2(K), k\chi_0)|\leq 2$ for all $k$. Lemma~\ref{evenintegerlemma} gives us that $\sigma_1(K, k\chi_0)$ is an integer divisible by 4 and so
$\sigma_1(\Sigma_2(K), k \chi_0)=0$. 

Now, let $\chi$ be a multiple of $\chi_0$ such that $\chi(\mu_p)=1$, $\chi(\mu_q)=b$, and so $\chi(\mu_r)=d-b-1$. 
So  we have equation $\eq(0):$
\[0=\frac{d^2}{2}\sigma_1(K,\chi)= (d-1)p+b(d-b)q+(b+1)(d-b-1)r. 
\]
We split into cases depending on the value of $b$.
\medskip

\noindent {\bf Case 1: $0<b<\frac{d-1}{2}$}:

Therefore $(2\chi)(\mu_p)=2$, $(2\chi)(\mu_q)=2b$, and $(2\chi)(\mu_r)=d-2b-2$. 
So we have equation $\eq(1)$:
\[ 0=\frac{d^2}{2}(\sigma_1(K,2 \chi))= 2(d-2)p+2b(d-2b)q+(2b+2)(d-2b-2)r\]
We then have that 
\begin{align*}
\frac{1}{2}(2\eq(0)-\eq(1))=p+b^2q+(b+1)^2r=0 \\
 \frac{1}{2d}(4\eq(0)-\eq(1))= p+bq+(b+1)r=0 
 \end{align*}
It follows that  $(b+1)r=-(b-1)q$ and finally that $p+q=0$, which is our desired contradiction. 
\smallskip

\noindent{\bf Case 2: $b=\frac{d-1}{2}$}:

In this case equation $\eq(0)$ implies that $q+r=-\frac{4p}{d+1}$. Also, $(2\chi)(\mu_p)=2$ and $(2\chi)(\mu_q)=(2\chi)(\mu_r)=d-1$, so we have equation $\eq(2)$:
\[0=2(d-2)p+(d-1)q+(d-1)r\]
Substituting our expression for $q+r$ into $\eq(2)$, we obtain that $(d^2-3d)p=0$, and so that $d=3$. But this implies that $q+r=-p$, and hence that $p$ is even, which is our desired contradiction. 
\smallskip

\noindent{\bf Case 3: $d/2<b<d$}:

Therefore $(2\chi)(\mu_p)=2$, $(2\chi)(\mu_q)=2b-d$, and $(2\chi)(\mu_r)= 2d-2b-2$. 
So we have equation $\eq(3)$:
\[0=\frac{d^2}{2}(\sigma_1(K, 2\chi))= 2(d-2)p+(2b-d)(2d-2b)q+ (2b-d+2)(2d-2b-2)r
\]
We then have that 
\begin{align*}
\frac{1}{2}(2\eq(0)-\eq(2))=p+(d-b)^2q+(d-b-1)^2r=0 \\
 \frac{1}{2d}(4\eq(0)-\eq(2))= p+(d-b)q+(d-b-1)r=0 
 \end{align*}
It follows that $(d-b)q=-(d-b-2)r$, and finally that $p+r=0$, which is our desired contradiction. 
\end{proof}

\subsection{Cases 4,5, and 6: $d$ divides $pq+pr+qr$ but not any of $p,q,r$}

The link $L_0$ considered as a 4-colored link has identically 0 colored signature, since it is a connected sum of 2-colored Hopf links. Note that since $d$ divides none of $p,q,$ and $r$, every nontrivial character $\chi$ to $\mathbb{Z}_d$ has all of $\chi(\mu_p),\chi(\mu_q),\chi(\mu_r),$ and $\chi(\mu_0)$ nonzero, and so Theorem~\ref{coloredlinks} applies and we have the following simple formula for $\sigma_1(\Sigma_2(K),\chi)$. 

\begin{prop}\label{fchiprop}
Let $K=P(p,q,r)$ and suppose $\chi:H_1(\Sigma_2(K)) \to \mathbb{Z}_d$ 
has $\chi(\mu_p),\chi(\mu_q),\chi(\mu_r),$ and $\chi(\mu_0)$ all nonzero. 
Let $a,b,c$, and $\epsilon$ be the unique lifts of $\chi(\mu_p),\chi(\mu_q),\chi(\mu_r),$ and $\chi(\mu_0)$ to $\{1, \dots, d-1\}$. 
Then 
\[\sigma_1(\Sigma_2(K),\chi)= 3- \frac{2}{d^2} f(\chi),
\]
where $f(\chi):=(d-\epsilon)(a+b+c) + (d-a)(ap+\epsilon)+ (d-b)(bq+\epsilon) + (d-c)(cr+ \epsilon).$
\end{prop}
Note that the parity of $a+b+c$ and of $\epsilon$ together determine the parity of $f(\chi)$; in particular, if $a+b+c$ is odd then $\epsilon$ and $f(\chi)$ have opposite parities. 
Also, when $a+b+c=d$ we have that 
\begin{align*}
f(\chi) &= d^2 + d \epsilon +a(d-a)p + b(d-b)q+ (a+b)(d-(a+b))r 
\end{align*}

\begin{lem}\label{divisibilitylemma}
Let $\chi: H_1(\Sigma_2(K)) \to \mathbb{Z}_d$, where $d$ divides none of $p,q,r$. Then $f(\chi)$ is divisible by $d^2$.
\end{lem}

\begin{proof}
First, recall that $H_1(\Sigma_2(K))$ is presented by linking matrix $A$, and so our $a,b,c,\epsilon$ values must satisfy
\[ a+b+c \equiv ap+\epsilon \equiv bq+\epsilon \equiv cr+\epsilon \equiv 0 \mod d.\] 
We can rewrite $f(\chi)$ as 
\begin{align*}
f(\chi)&=d\left[ (a+b+c)+ (ap+\epsilon) + (bq+ \epsilon) + (cr+ \epsilon)\right] \\
    & \quad \quad - \left[\epsilon(a+b+c)+ a(ap+\epsilon)+ b(bq+\epsilon)+c(cr+\epsilon)\right].
\end{align*}
The first term can immediately be seen to be divisible by $d^2$, and so it suffices to show that $g(\chi)=\epsilon(a+b+c)+ a(ap+\epsilon)+ b(bq+\epsilon)+c(cr+\epsilon)$ is also divisible by $d^2$. Writing $ap+\epsilon=k_1d$, $bq+\epsilon=k_2d$, and $cr+\epsilon=k_3d$ for $k_1, k_2, k_3 \in \mathbb{Z}$, we have
\begin{align*}
g(\chi)&= a(ap+\epsilon +\epsilon)+ b(bq+\epsilon +\epsilon)+c(cr+\epsilon+ \epsilon) \\
&= \frac{k_1 d - \epsilon}{p} (k_1d+\epsilon)+\frac{k_2 d - \epsilon}{q} (k_2d+\epsilon) +\frac{k_3 d - \epsilon}{r} (k_3d+\epsilon) \\
&= \frac{k_1^2d^2- \epsilon^2}{p}+\frac{k_2^2d^2- \epsilon^2}{q}+\frac{k_3^2d^2- \epsilon^2}{r}
\end{align*}
Note that since $d$ is relatively prime to all of $p,q,$ and $r$, we can multiply through by $pqr$ without changing the divisibility of $g(\chi)$ by $d^2$. We therefore have the desired result, since
\begin{align*}
g(\chi)pqr&= (k_1^2d^2- \epsilon^2)qr+(k_2^2d^2- \epsilon^2)pr+(k_3^2d^2- \epsilon^2)pq \\
&= d^2(k_1^2qr+ k_2^2qr+k_3^2pr) -(pq+qr+pr)\epsilon^2. 
\qedhere
\end{align*}
\end{proof}

\begin{prop}[Case 4]\label{generaldneqdividepqr}
Let $K=P(p,q,r)$ with $q \geq p>0$ and $r<0$ and $d$ be some prime dividing $pq+pr+qr$ which divides none of $p,q$ and $r$.
Suppose also that $r \neq -(4p+q)$ and that $p \not \equiv q$ mod $d$.
Then the Casson-Gordon signatures of $\Sigma_2(K)$ associated to characters to $\mathbb{Z}_d$ obstruct $K$'s topological sliceness. 

\end{prop}

\begin{proof} 
Assume for the sake of contradiction that $K$ is CG-slice at $d$. Since $H_1(\Sigma_2(K), \mathbb{Z}_d)$ is cyclic, for any $\chi: H_1(\Sigma_2(K)) \to \mathbb{Z}_d$ we must have 
\[ \left|\sigma_1(\Sigma_2(K), \chi)\right|=\left|3- \frac{2}{d^2}f(\chi)\right| \leq 1.\]
and so by Lemma~\ref{divisibilitylemma} we have $f(\chi)= d^2 \text{ or } 2d^2$. 

We will work with two characters. Note that our formula for $f(\chi)$ uses the unique integer lifts of $\chi(\mu_i)$ to $\{1, \dots, d-1\}$, so we will be careful to only write $\chi(\mu_i)=x$ if $0<x<d$. 
We define $\chi_1$ to have $\chi_1(\mu_r)=1$, and $\chi_2=2\chi_1$. It follows that $\chi_1(\mu_0)$ is the unique integer $\epsilon$ in $(0,d)$ such that $\epsilon + r \equiv 0 \mod d$, $\chi_1(\mu_p)$ is the unique integer $a$ in $(0,d)$ such that $\epsilon + ap \equiv 0 \mod d$, and $\chi_1(\mu_q)= d-a-1$. 
Note that $\chi_i(\mu_p)+ \chi_i(\mu_q)+\chi_i(\mu_r)=d$, so $f(\chi_i)$ has the opposite parity as $\chi_i(\mu_0)$ for $i=1,2$. We now  define some convenient notation:
\begin{align*}
\left[\begin{array}{c} x_1 \\ x_2 \end{array} \right]_y=
\left\{\begin{array}{l} x_1  \text{ if } 0<y<\frac{d}{2} \\
x_2 \text{ if } \frac{d}{2}<y<d \end{array}\right. 
\text{ and }
\left[\begin{array}{c} x_1 \\ x_2 \end{array} \right]_{p(y)}=
\left\{\begin{array}{l} x_1  \text{ if $y$ is even} \\
x_2 \text{ if $y$ is odd} \end{array}\right. .
\end{align*}
We therefore have\footnote{Note that if $a=\frac{d-1}{2}$  then $\chi_1$ sends both $\mu_p$ and $\mu_q$ to $\frac{d-1}{2}$. But this implies that $p \equiv q \mod d$, which we have assumed is not the case.} $\chi_2(\mu_p)= \left[\begin{array}{c} 2a \\ 2a-d \end{array} \right]_a$, $\chi_2(\mu_q)= \left[\begin{array}{c} d-2a-2\\ 2d-2a-2 \end{array} \right]_a$,
$\chi_2(\mu_0)=\left[\begin{array}{c} 2\epsilon \\ 2\epsilon-d \end{array} \right]_\epsilon$, $f(\chi_1)=\left[\begin{array}{c} d^2 \\ 2d^2 \end{array} \right]_{p(\epsilon)}$, and $f(\chi_2)=\left[\begin{array}{c} d^2 \\ 2d^2 \end{array} \right]_\epsilon$.  
We therefore have the following two equations coming from our formulas for $f(\chi_1)$ and $f(\chi_2)$:
\begin{align}
 &\left[\begin{array}{c} 0 \\ d^2 \end{array} \right]_{p(\epsilon)}=
d\epsilon + a(d-a)p + (a+1)(d-a-1)q+(d-1)r
\\
&\left[\begin{array}{c} 0 \\ d^2 \end{array} \right]_\epsilon=
d\epsilon+
 \left[\begin{array}{c} a(d-2a)p + (a+1)(d-2a-2)q \\ (2a-d)(d-a)p+(2+2a-d)(d-a-1)q \end{array} \right]_a +(d-2)r
\end{align} 
Consider $\eq(3)=\eq(1)-\eq(2)$ and $\eq(4)= \frac{1}{d} \left(2\eq(1)-\eq(2)\right)$:
\begin{align}
 \left[\begin{array}{c} 0 \\ d^2 \end{array} \right]_{p(\epsilon)}-\left[\begin{array}{c} 0 \\ d^2 \end{array} \right]_\epsilon
=\left[\begin{array}{c} a^2p+(a+1)^2q \\ (d-a)^2p + (d-a-1)^2q \end{array}\right]_a+r
\\
\left[\begin{array}{c} 0 \\ 2d \end{array} \right]_{p(\epsilon)} -\left[\begin{array}{c} 0 \\ d \end{array} \right]_\epsilon
=
\epsilon+\left[\begin{array}{c} ap+(a+1)q\\(d-a)p+(d-a-1)q\end{array}\right]_a+r
\end{align}
Note that the left side of $\eq(4)$ is even exactly when $\epsilon<d/2$, while the right side has the same parity as $\epsilon$. 
So we can assume $\epsilon<d/2$ if and only if $\epsilon$ is even, and  $\eq(3)$ and $\eq(4)$ simplify to the following:
 \begin{align}
0=\left[\begin{array}{c}a^2p+(a+1)^2q \\ (d-a)^2p + (d-a-1)^2q \end{array}\right]_a +r
\end{align}

\begin{align}
\left[\begin{array}{c} 0 \\ d \end{array} \right]_\epsilon
=
\epsilon+ \left[\begin{array}{c} ap+(a+1)q\\(d-a)p+(d-a-1)q\end{array}\right]_a +r
\end{align}

We can use Equation~(5) to see that if $a<d/2$ then $D=ap+(a+1)q$ and if $a>d/2$ then $D=(d-a)p+(d-a-1)q$. 
We will now split into cases, and show that each leads to a contradiction by using Equation (5) to write $r$ in terms of $a,d,p,q$ and substituting this expression into Equation (6). Note that since $d$ divides $D$, we certainly have that $d \leq D$. 
\medskip

\noindent {\bf Case 1:} $a, \epsilon< d/2$.

 A little rewriting gives that $\epsilon= a^2(p+q)+(q-p)$, and so that  
\[2a^2(p+q)<2a^2(p+q)+2(q-p)= 2 \epsilon<d \leq D= ap+(a+1)q.\]
It follows that $(2a^2-a)p+(2a^2-a-1)q<0$, which gives the desired contradiction. 
\smallskip

\noindent{\bf Case 2:} $a<d/2<\epsilon$.
\[0<d-\epsilon= -a(a-1)p-a(a+1)q<0 \]

\noindent{\bf Case 3:} $\epsilon<d/2<a$.

First, suppose that $a=d-2$. Then Equation (5) implies that $r=-(4p+q)$, which we have assumed is not the case. So we can assume that $a<d-2$, and so
\[ D=(d-a)p+(d-a-1)q< (d-a)(d-a-1)p+(d-a-1)(d-a-2)q= \epsilon<d. \]

\noindent{\bf Case 4:} $d/2<a,\epsilon$.

As in Case 3, we can assume that $a<d-2$ and so
\[0<d-\epsilon=-(d-a)(d-a-1)p-(d-a-1)(d-a-2)q<0. \qedhere \]
\end{proof}

\begin{prop}[Case 5]\label{exceptionalprop}
Suppose $K=P(p,q,r)$ for $r=-(4p+q)$. Suppose $d$ is a prime that divides $pq+pr+qr$ but none of $p,q,$ and $r$.  
Then either $K=P(1,q,-(q+4))$, in which case $K$ is ribbon, or the Casson-Gordon signatures of $\Sigma_2(K)$ corresponding to characters to $\mathbb{Z}_d$ obstruct $K$'s topological sliceness.
\end{prop}
Note that $K=P(1,q, -(q+4))$ is a 2-bridge knot. If we write $q=2k+1$, then $K$ is a generalized twist knot corresponding to the fraction $-\frac{4(k+1)(k+2)+1}{2(k+1)}$ and has been known to be ribbon at least since \cite{CG78}. 
\begin{proof}
Let $\chi$ be the character sending $\mu_p$ to $d-2$, $\mu_q$ and $\mu_r$ to 1, and $\mu_0$ to $\epsilon$. Then $\chi'= \frac{d-1}{2} \chi$ sends $\mu_p$ to 1, $\mu_q$ and $\mu_r$ to $\frac{d-1}{2}$, and $\mu_0$ to $\epsilon'$. Arguments as in the proof of Proposition~\ref{generaldneqdividepqr} show that if $p>1$ then at least one of $|\sigma_1(\Sigma_2(K),\chi)|$ and $|\sigma_1(\Sigma_2(K),\chi')|$ is strictly larger than 1, and hence that $K$ is not CG slice at $d$. 
%
%
%
%
\end{proof}

\begin{prop}[Case 6]\label{chipischiqprop}
Suppose that $d$ divides $pq+pr+qr$ but none of $p$, $q$, and $r$, that
 $p \equiv q \mod d$, and that $d\neq3$. 
Then the Casson-Gordon signatures of $\Sigma_2(K)$ associated to characters to $\mathbb{Z}_d$ obstruct $K$'s topological sliceness. 
\end{prop}
\begin{proof} 
For $i=1,2$, consider the characters $\chi_i:H_1(\Sigma_2(K)) \to \mathbb{Z}_d$ defined by $\chi_i(\mu_p)=\chi_i(\mu_q)=i$, $\chi_i(\mu_r)=d-2i$, and $\chi_i(\mu_0)=\epsilon_i$. (Note that since $d \neq 3$ we have that $d-2i>0$ for $i=1,2$.) Arguments as in the proof of Proposition~\ref{generaldneqdividepqr} show that at least one of $|\sigma_1(\Sigma_2(K),\chi_i)|$ is strictly larger than 1, and hence that $K$ is not CG-slice at $d$. 
%
%
%
%
\end{proof}

\begin{prop} \label{powerof3}
Suppose that $K=P(p,q,r)$ has $p,q,$ and $r$ relatively prime, $|H_1(\Sigma_2(K)|=|pq+pr+qr|=3^{2n}$ for some $n \in \mathbb{N}$, and $p \equiv q \mod 3$.  
Then either $K$ is ribbon or the Casson-Gordon signatures associated to characters of order $3$ and $9$ obstruct $K$'s topological sliceness. 
\end{prop}

\begin{proof}
First, suppose that $n \geq 2$.  Since $p,q,$ and $r$ are pairwise relatively prime, $H_1(\Sigma_2(K))$ is cyclic, and any character to $\mathbb{Z}_{3^n}$ will vanish on the unique metabolizer for the linking form. Proposition~\ref{qhs3recognition} implies that the associated covers are rational homology spheres, and so it suffices to find such a character $\chi$ with $|\sigma_1(\Sigma_2(K),\chi)|>1$. 
The arguments of Propositions~\ref{generaldneqdividepqr}, \ref{exceptionalprop}, and \ref{chipischiqprop} applied to $d=9$ (according to whether$r=-(4p+q)$ and whether $p \equiv q \mod 9$) show that this is the case. 

Now suppose that $n=1$ and so $pq+pr+pq=-9$ and $r=-\frac{pq+9}{p+q}$. A slight variation on our usual arithmetic arguments then implies that $\sigma_1(\Sigma_2(K),\chi)<-1$ for some $\chi:H_1(\Sigma_2(K)) \to \mathbb{Z}_3$, and hence that $K$ is not CG-slice at $d=3$. 
%
%
%
\end{proof}

\section{Topological sliceness of even 3-strand pretzel knots.}

We now outline the proof of our argument that all topologically slice even 3-strand pretzel knots are either ribbon or in Lecuona's family $\{\pm P_a\}$, leaving the details of arithmetic to the reader. 

\begin{thm}
Let $K$ be an even 3-strand pretzel knot. 
Suppose that $K$ is topologically slice. 
Then, up to reflection, either $K=P(p,-p,q)$ for some $p,q \in \mathbb{N}$ (and $K$ is ribbon) or $K=P_a=P\left(a, -a-2, -\frac{(a+1)^2}{2}\right)$ for some $a \equiv 1, 11, 37, 47, 59 \mod 60$. 
\end{thm}

\begin{proof}
Suppose that $K$ is an algebraically slice even 3-strand pretzel. First, note that by Jabuka's computation of the rational Witt classes of pretzel knots, we can assume that either $K=P(p,-p,q)$ for some odd $p$ and even $q$ or that $K=P(-p,p\pm2, q)$ for some odd $p$ and even $q$ such that $\det(K)= \pm2q - p^2 \mp 2p=m^2>0$ (Theorem 1.11 of \cite{Jab10}). In the first case $K$ is ribbon, and so we assume that we are in the second case. By the symmetries of 3-strand pretzel knots, we can also assume that up to reflection $K= P(-p,p+2,q)$ for $p\in \mathbb{N}$. Then our condition that $\det(K)=2q-p^2-2p>0$ implies that $q>0$ as well. 

First, observe that if $\det(K)=1$ then $q= \frac{(p+1)^2}{2}$ and $K$ is an element of Lecuona's family $\{P_a\}$. For $a \not \equiv 1, 11, 37,47,49,59 \mod 60$, Theorem~4.5 of \cite{Lec13} states that $K$ is not algebraically slice. When $a \equiv 49 \mod 60$, an argument analogous to the proof of in Theorem~4.5 of \cite{Lec13} shows that $\Delta_K(t)$ does not have a Fox-Milnor factorization and hence that $K$ is not algebraically slice.\footnote{
In particular, note that since $a  \equiv 49 \mod 60$ we have that $5$ divides $\frac{(a+1)^2}{4}$ 
and $3$ divides $a+2$. Working mod $5$, $\Delta_{P_a}(t)\equiv \Pi_{1 \neq d | a} \Phi_{d}(t) \Pi_{1 \neq d|a+2} \Phi_{d}(t)$, where $\Phi_{d}(t)$ denotes the $d^{th}$ cyclotomic polynomial. Since $\Phi_3(t)$ is symmetric, irreducible mod $5$, and relatively prime to each $\Phi_d(t)$
for $d \neq 3$ dividing $a$ or $a+2$, the desired result follows.}

So we can assume that $\det(K)=m^2>1$, and in particular that there is an (odd) prime $d$ dividing $\det(K)$. Arguments as in the proof of Proposition~\ref{dividespq} show that $\Sigma_2(K)$ has a surgery presentation with link the coherently oriented torus link $-T(2,2p)$ and linking matrix $\left[ \begin{array}{cc} 2 & -p \\ -p & q-p \end{array} \right]$. It follows that $H_1(\Sigma_2(K))$ is cyclic, and hence that $H_1(\Sigma_2(K), \mathbb{Z}_d)$ is certainly cyclic as well.
It therefore suffices to show that there is a single $\chi:H_1(\Sigma_2(K)) \to \mathbb{Z}_d$ with $|\sigma_k(\Sigma_2(K),\chi)|>1$ for some $1\leq k<d$. 

The construction of $\chi$ and computation of the corresponding Casson-Gordon signatures is extremely similar to the arguments of Section~\ref{thework}, and therefore we only enumerate the cases one must consider, and leave the verification of the details to the interested reader.
It is convenient to consider six cases, according to the values of $K$'s parameters mod $d$: $-p \equiv q \equiv 0 $; $p+2 \equiv q \equiv 0$;  
$-p \equiv 2q \not \equiv 0 $; $p+2 \equiv 2q \not \equiv 0 $; $-p \equiv p+2 \not \equiv 0$; and $-p, p+2$, and $q$ are mutually distinct and nonzero.
\end{proof}
 
\section*{Acknowledgements}
I would like to thank Tye Lidman for suggesting Corollary~\ref{subgroup}, Lisa Piccirillo for her help with manipulating surgery descriptions, and my advisor Cameron Gordon for his advice, support, and very helpful comments.

\bibliography{cgpretzel}

\end{document}